\documentclass{amsart}

\usepackage[
  hmarginratio={1:1},     
  vmarginratio={1:1},     
  textwidth=400pt,        
  heightrounded,          
]{geometry}

\usepackage{amssymb,amsmath,amsthm,latexsym,amscd}
\usepackage{epsfig}
\usepackage{psfrag}

\newtheorem{Theorem}{\bf Theorem}
\newtheorem{lemma}[Theorem]{\bf Lemma}
\newtheorem{proposition}[Theorem]{\bf Proposition}
\newtheorem{corollary}[Theorem]{\bf Corollary}

\newtheorem{definition}[Theorem]{\bf Definition}

\newtheorem{remark}[Theorem]{\bf Remark}
\newtheorem{theorem}[Theorem]{\bf Theorem}

\def\qed{\hfill$\Box$}
\newcommand{\be}{\begin{equation}}
\newcommand{\ee}{\end{equation}}

\def\hpic #1 #2 {\mbox{$\begin{array}[c]{l} 
\epsfig{file=#1,height=#2}\end{array}$}}
\def\wpic #1 #2 {\mbox{$\begin{array}[c]{l} 
\epsfig{file=#1,width=#2}\end{array}$}}

\begin{document}

\title[Crossed products and the Drinfeld double]{Note on infinite iterated crossed products of Hopf algebras and the Drinfeld double}

\author{Sandipan De}
\author{Vijay Kodiyalam}
\address{The Institute of Mathematical Sciences, Chennai, India}
\email{sandipande@imsc.res.in,vijay@imsc.res.in}




\begin{abstract} 
For a finite dimensional Hopf algebra we show that
an associated natural inclusion of infinite crossed products is the crossed product by the Drinfeld double, and that this
is a characterisation of the double.
\end{abstract}
\maketitle

\section{introduction}

The motivation for this paper comes from a series of talks
delivered by Prof. Masaki Izumi at the Institute of Mathematical Sciences, Chennai in March 2012, during one
of which he asserted that for a Kac algebra subfactor, a
related subfactor to its asymptotic inclusion comes from
an outer action of its Drinfeld double. This is a folklore
result in subfactor theory and since we were unable to locate a precise reference, it seemed desirable to write
down a formal proof. However, in the process of trying to prove this,
we noticed a purely algebraic result that also seemed quite
interesting and it is this algebraic result that is the focus
of this paper (rather than the original subfactor result, which involves
more analysis and to which we hope to return in a future publication).
One advantage of this algebraic approach is that it applies to Hopf algebras
that are not semisimple, in contrast to the analytic case.

Our main result is roughly the following. Given a finite dimensional Hopf algebra $H$, we associate a certain
natural inclusion of (infinite-dimensional) algebras $A \subseteq B$ to it and show that $B$ is the crossed product (also known as the smash product in Hopf algebra literature) of $A$ by $D(H)$ where $D(H)$ is the Drinfeld double of $H$. More significantly, we show that $D(H)$ is the only finite-dimensional Hopf algebra with this property thus producing a context in which the Drinfeld double arises very naturally.

\section{Iterated crossed products}

Throughout this paper, 
we work over a fixed but arbitrary ground field $k$.
All algebras in this paper will be unital $k$-algebras and possibly infinite-dimensional. However, the Hopf algebras we consider will always be finite-dimensional. Subalgebras will always refer to unital subalgebras.

We will assume familiarity
only with the basics of Hopf algebras as in \cite{Kss1995} or in \cite{Mjd2002}, as also the existence and uniqueness of integrals for finite-dimensional Hopf algebras for which we
refer to the beautiful pictorial treatment in \cite{KffRdf2000},
which offers a detailed exposition of the formalism of \cite{Kpr1996}.
Our treatment is otherwise self contained.

Suppose that $A$ is an algebra and $H = (H,\mu,\eta,\Delta,\epsilon,S)$ is a finite-dimensional Hopf algebra. By an action of $H$ on $A$
we will mean a linear map $\alpha : H \rightarrow End(A)$
(references to endomorphisms without further qualification will be to $k$-linear endomorphisms)
satisfying (i) $\alpha_1 = id_A$, (ii) $\alpha_{xy} =\alpha _x \circ \alpha _y$, (iii) $\alpha_x(1_A) = \epsilon(x) 1_A$ and
(iv) $\alpha_x(ab) = \alpha_{x_1}(a)\alpha_{x_2}(b)$, for all $x,y \in H$ and $a,b \in A$. To clarify notation, $\alpha_x$ stands for $\alpha(x)$ and
$\Delta(x)$ is denoted by $x_1 \otimes x_2$ (a simplified version of the Sweedler coproduct notation).

We draw the reader's attention to a notational abuse of which we will often be guilty. We denote elements of
a tensor product as decomposable tensors with the understanding that there is an implied omitted summation
(just as in our simplified Sweedler notation).
Thus, when we write `suppose $f \otimes x \in H^* \otimes H$', we mean `suppose $\sum_i f^i \otimes x^i \in H^* \otimes H$' (for some $f^i \in H^*$ and $x^i \in H$, the sum over a finite index set).

Given an action of $H$ on $A$, we may define the crossed product algebra (or the smash product algebra) denoted $A \rtimes_\alpha H$ (or mostly, simply as $A \rtimes H$, when the action is understood) to be the algebra with
underlying vector space $A \otimes H$ (where we denote $a \otimes x$ by $a \rtimes x$) and multiplication
defined by 
$$
(a \rtimes x)(b \rtimes y) = a\alpha_{x_1}(b) \rtimes x_2y.
$$
This is an algebra with unit $1_A \rtimes 1_H$ and
there are natural inclusions of algebras $A \subseteq A \rtimes H$ given by $a \mapsto a \rtimes 1_H$ and $H \subseteq A \rtimes H$ given by $x \mapsto 1_A \rtimes x$.
We note that while the crossed or smash product construction is a special case of one that
involves, in addition, twisting by a 2-cocyle of $H$, in this paper, it suffices
to consider the case of the trivial cocycle, which is the one discussed above.

Borrowing terminology from subfactor theory, we define 
an inclusion $A \subseteq B$ of algebras to be irreducible if the relative commutant $A^\prime \cap B$ (which is the centraliser algebra of $A$ in $B$, also denoted by
$C_B(A)$ or $B^A$) 
is just $k1_B$, i.e., if the only elements of $B$ that commute with all elements of $A$ are scalar multiples of its identity element. We also define an action of $H$ on $A$ to be outer if
the inclusion $A \subseteq A \rtimes H$ is irreducible.

The following lemma, whose proof we omit,  is a simple and useful characterisation of crossed products without explicit reference to an action.
We will say that two algebras
containing an algebra $A$ are isomorphic as algebras over $A$, if they are isomorphic by an isomorphism that restricts to the identity on $A$.

\begin{lemma}\label{simple} Suppose that $B$ is an algebra with subalgebras $A$ and $H$, where $H$ is further equipped with a comultiplication and antipode that make it a Hopf algebra, and such that:
\begin{enumerate}
\item[(i)] The restriction of the multiplication map $\mu : A \otimes H \rightarrow B$ is a linear isomorphism, and
\item[(ii)] For all $x \in H$ and $a \in A$, $x_1aSx_2 \in A.$
\end{enumerate}
Then $\alpha: H \rightarrow End(A)$ defined by $\alpha_x(a) = x_1aSx_2$ is an action of $H$ on $A$ and the crossed product algebra $A \rtimes_\alpha H$ is isomorphic to $B$ as an algebra over $A$. \qed
\end{lemma}

Recall that if $\alpha$ is an action of a Hopf algebra $H$ on an algebra $A$, then there is a natural action $\beta$ of the dual Hopf algebra $H^*$ on the crossed product algebra $A \rtimes H$ defined by 
$\beta_f(a \rtimes x) = f(x_2)(a \rtimes x_1)$. In the sequel,
we will use this action without further specification (and in particular, for the action of $H^*$ on $H$ and that of $H$ on $H^*$).

We now review infinite iterated crossed products. Our treatment
closely follows that of \cite{JjoSnd2009} which treats the
case when $H$ is a Kac algebra.
For $i \in {\mathbb Z}$, define $H^i$ to be
$H^*$ or $H$ according as $i$ is even or odd. For $i \leq j$ define $H^{[i,j]}$ by induction on $j-i$ as $H^i$ if $j=i$
and as $H^{[i,j-1]} \rtimes H^j$ otherwise (for the natural action). The multiplication on $H^{[i,j]}$ is seen (by induction) to be
given by the following formula (when $i,j$ are both even - similar
formulae hold for the other three cases):
\begin{eqnarray*}
(f^i \rtimes x^{i+1} \rtimes \cdots \rtimes f^j)(g^i \rtimes y^{i+1} \rtimes \cdots \rtimes g^j) = \\
\langle x^{i+1}_1 | g^i_2 \rangle
\langle f^{i+2}_1 | y^{i+1}_2 \rangle
\cdots
\langle x^{j-1}_1 | g^{j-2}_2 \rangle
\langle f^j_1 | y^{j-1}_2 \rangle \times \\
f^ig^i_1 \rtimes x^{i+1}_2y^{i+1}_1 \rtimes \cdots \rtimes x^{j-1}_2y^{j-1}_1 \rtimes f^j_2g^j.
\end{eqnarray*}
The multiplication is given pictorially in Figure \ref{fig:multfig}.\begin{figure}[htbp]
\psfrag{fi}{\small $f^i$}
\psfrag{gi}{\small $g^i$}
\psfrag{xip1}{\small $x^{i+1}$}
\psfrag{yip1}{\small $y^{i+1}$}
\psfrag{fip2}{\small $f^{i+2}$}
\psfrag{vdots}{$\vdots$}
\psfrag{gjm2}{\small $g^{j-2}$}
\psfrag{xjm1}{\small $x^{j-1}$}
\psfrag{yjm1}{\small $y^{j-1}$}
\psfrag{fj}{\small $f^j$}
\psfrag{gj}{\small $g^j$}
\psfrag{11}{\tiny $(6,1)$}
\psfrag{12}{\tiny $(6,2)$}
\begin{center}
\includegraphics[height=1.5in]{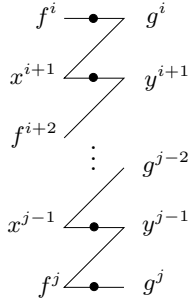}
\caption{Multiplication in $H^{[i,j]}$ for $i,j$ even}
\label{fig:multfig}
\end{center}
\end{figure}
The interpretation of Figure \ref{fig:multfig} is as follows. The dots are to be interpreted as multiplication (in $H$ or in $H^*$), the diagonal lines
as contractions (between $H$ and $H^*$ to get a constant) and the forks as applications of $\Delta$.
Note that $H^{[i,i+1]} = H^i \rtimes H^{i+1}$ is known as the Heisenberg double of $H^i$ (and is isomorphic to a matrix algebra of size $dim(H)$).

The multiplication rule shows that if $p \leq i \leq j \leq q$, the natural inclusion of $H^{[i,j]}$ into $H^{[p,q]}$ is an algebra map. Define the algebra $B$ to be the `union' of all the $H^{[i,j]}$. More precisely, $B$ is the direct limit, over the subset of finite intervals in ${\mathbb Z}$ directed by inclusion,  of the $H^{[i,j]}$.
We may suggestively write $B = H^{(-\infty,\infty)} = \cdots \rtimes H \rtimes H^* \rtimes H \rtimes \cdots$ and represent a typical element of $B$ as $\cdots \rtimes x^{-1} \rtimes f^0 \rtimes x^1 \rtimes \cdots$. We repeat that this means that a typical element of $B$ is in fact a finite sum of such terms. Note that in any such term all but finitely
many of the $f^i$ are $\epsilon$ and all but finitely many of the $x^i$ are $1$. One fact about the infinite iterated crossed product that we will use is that $H^i$ and $H^j$ commute whenever $|i-j| \geq 2$.

Next, we define a subalgebra $A$ of $B$ which, in suggestive notation, is $H^{(-\infty,-1]} \otimes H^{[2,\infty)}$.
A little more clearly, it consists of all (finite sums of) elements $\cdots \rtimes x^{-1} \rtimes f^0 \rtimes x^1 \rtimes \cdots$ of $B$ with $f^0 = \epsilon$ and $x^1 = 1$. Strictly speaking,
if $H^{(-\infty,-1]}$ represents the direct limit of all the
$H^{[-j,-1]}$ for $j \geq 1$ and $H^{[2,\infty)}$ represents the direct limit of all the $H^{[2,j]}$ for $j \geq 2$, then,
these algebras can be identified with commuting subalgebras of $B$,
with the multiplication map being an injective map from
$H^{(-\infty,-1]} \otimes H^{[2,\infty)}$ to $B$, and the image is denoted $A$. As an algebra, $A$ is clearly generated by all the $H^i$ for $i \in {\mathbb Z} \backslash \{0,1\}$.
 
The main object of interest in this paper is the following pair of algebras.

\begin{definition}
For a finite-dimensional Hopf algebra $H$, the inclusion
$A \subseteq B$ of (infinite-dimensional) algebras defined above will be called the derived pair of $H$.
\end{definition}

The following proposition identifying the relative commutant of the derived pair will be very useful. In case $H$ is a Kac algebra, this appears in \cite{Jjo2008}.

\begin{proposition}\label{commutants}
For any $p\in {\mathbb Z}$, the subalgebras $H^{(-\infty,p]}$ and $H^{[p+2,\infty)}$ are mutual commutants in $B$. In particular, the derived pair of $H$ is irreducible.
\end{proposition}

The main observation in the proof of Proposition \ref{commutants} is contained in the following lemma.
Recall that a left integral of $H$ is an element $h \in H$ that satisfies $xh = \epsilon(x)h$ for all $x \in H$. The space of left integrals of $H$ is $1$-dimensional and if $h \in H$ is 
a left integral, then 
$S^{-1}h_1x \otimes h_2 = S^{-1}h_1 \otimes xh_2$ for all $x \in H$.
Further, the (Fourier transform) map $f \mapsto f(h_2)h_1 : H^* \rightarrow H$ is a linear isomorphism (and in fact, a $H^*$-module isomorphism for the natural $H^*$-module structure
on $H$). 
%
%
%
%
%
%
%

\begin{lemma}\label{commlemm}
For $i \leq j$, the set of elements of $H^{[i,j]}$ that commute with a non-zero left integral in $H^{i-1}$ is precisely $H^{[i+1,j]}$.
\end{lemma}

\begin{proof} First suppose that $i$ is even, so that $H^{i-1} = H$. 
Since elements of $H^{[i+1,j]}$ certainly commute with all elements of $H^{i-1}$, it suffices to see that an arbitrary element, say $f^i \rtimes x^{i+1} \rtimes \cdots \in H^{[i,j]}$ that commutes with a non-zero left integral, say $h^{i-1} \in H^{i-1}$, is actually in
$H^{[i+1,j]}$.

The commutativity condition is equivalent to the equation
$$
h^{i-1} \rtimes f^i \rtimes x^{i+1} \rtimes \cdots = f^i_1(h^{i-1}_2) h^{i-1}_1 \rtimes f^i_2 \rtimes x^{i+1} \rtimes \cdots.
$$
Comparing coefficients of a basis of $H^{[i+1,j]}$ on both sides, we get
$h^{i-1} \rtimes f^i  = f^i_1(h^{i-1}_2) h^{i-1}_1 \rtimes f^i_2$. Evaluating the
second component on $1$ gives $f^i(1)h^{i-1} = f^i(h^{i-1}_2)h^{i-1}_1$.
But now, since $f^i(1)h^{i-1} = f^i(1) \epsilon(h^{i-1}_2)h^{i-1}_1$, the injectivity of the Fourier transform map implies that $f^i = f^i(1) \epsilon$.
Therefore $f^i \rtimes x^{i+1} \rtimes \cdots \in H^{[i+1,j]}$, as desired.

A similar proof is valid if $i$ is odd, replacing $H$ with $H^*$.
%
\end{proof}

\begin{proof}[Proof of Proposition \ref{commutants}]
Since $H^i$ and $H^j$ commute for $|i-j| \geq 2$, the subalgebras $H^{(-\infty,p]}$ and $H^{[p+2,\infty)}$ of $B$ commute with each other and therefore are contained in
the commutants of one another.

To show that $(H^{(-\infty,p]})^\prime \subseteq H^{[p+2,\infty)}$, take $1 \neq b \in (H^{(-\infty,p]})^\prime$, 
 and choose $i$ largest so that $b \in H^{[i,j]}$ for some $j$. It suffices to see that $i \geq p+2$. Suppose that $i \leq p+1$ so that $i-1 \leq p$. Now $b \in H^{[i,j]}$ and commutes with $H^{i-1}$ (since $H^{i-1} \subseteq H^{(-\infty,p]}$). By Lemma \ref{commlemm} it follows that
$b \in H^{[i+1,j]}$ contradicting choice of $i$.

To see that $(H^{[p+2,\infty)})^\prime \subseteq H^{(-\infty,p]}$, note that the `flip map about $p+1$' from $B$ to $B$ defined by  
\begin{eqnarray*}
\cdots  \rtimes f^{p-1} \rtimes x^p \rtimes f^{p+1} \rtimes x^{p+2}  \rtimes f^{p+3} \rtimes\cdots &\\
\mapsto \cdots  \rtimes Sf^{p+3} \rtimes S^{-1}x^{p+2} \rtimes Sf^{p+1} \rtimes S^{-1}x^{p}  \rtimes Sf^{p-1} \rtimes \cdots
\end{eqnarray*}
(for $p$ odd, with  a similar definition for $p$ even)
is an
anti-automorphism that interchanges $H^{[p+2,\infty)}$ and $H^{(-\infty,p]}$ and appeal to the previously proved case.

Finally, to see irreducibility of the derived pair, note that
$A^\prime \cap B = (H^{(-\infty,-1]})^\prime \cap (H^{[2,\infty)})^\prime = H^{[1,\infty)} \cap H^{(-\infty,0]} = k1_B$ - as desired.
\end{proof}

\section{The Drinfeld double construction}

We next review the Drinfeld double construction from
\cite{Mjd2002}. The Drinfeld double of a Hopf algebra $H$, denoted $D(H)$, is the Hopf algebra whose underlying vector space is $H^* \otimes H$ and multiplication, comultiplication and antipode specified by the following formulae.
\begin{eqnarray*}
(f \otimes x)(g \otimes y) &=& g_1(Sx_1)g_3(x_3) (g_2f \otimes x_2y),\\
\Delta(f \otimes x) &=& (f_1 \otimes x_1) \otimes (f_2 \otimes x_2), {\mbox { and}}\\
S(f \otimes x) &=& f_1(x_1)f_3(Sx_3) (S^{-1}f_2 \otimes Sx_2).
\end{eqnarray*}
What we actually use is an isomorphic avatar of this, which we denote $\tilde{D}(H)$ which also has underlying vector space $H^* \otimes H$ and structure maps obtained by transporting the structures on $D(H)$ using the invertible map $S \otimes S^{-1} : D(H) = H^* \otimes H \rightarrow H^* \otimes H = \tilde{D}(H)$. It is easily checked that the structure maps for $\tilde{D}(H)$ are given by the following formulae.
\begin{eqnarray*}
(f \otimes x)(g \otimes y) &=& g_1(x_1)g_3(Sx_3) (fg_2 \otimes yx_2),\\
\Delta(f \otimes x) &=& (f_2 \otimes x_2) \otimes (f_1 \otimes x_1), {\mbox { and}}\\
S(f \otimes x) &=& f_1(Sx_1)f_3(x_3) (S^{-1}f_2 \otimes Sx_2).
\end{eqnarray*}
The Hopf algebra $\tilde{D}(H)^{cop}$ is the Hopf algebra
(also with underlying space $H^* \otimes H$) with structure maps given by:
\begin{eqnarray*}
(f \otimes x)(g \otimes y) &=& g_1(x_1)g_3(Sx_3) (fg_2 \otimes yx_2),\\
\Delta(f \otimes x) &=& (f_1 \otimes x_1) \otimes (f_2 \otimes x_2), {\mbox { and}}\\
S(f \otimes x) &=& f_1(Sx_1)f_3(x_3) (Sf_2 \otimes S^{-1}x_2).
\end{eqnarray*}
For ease of notation we will denote the Hopf algebra $\tilde{D}(H)^{cop}$ by $L$. By construction, as a Hopf algebra, it is isomorphic to $D(H)^{cop}$.

\begin{lemma}\label{hop}
$\tilde{D}(H^{cop}) \cong \tilde{D}(H)^{cop}$ as Hopf algebras.
\end{lemma}

\begin{proof} It follows from the above that the structure maps for $\tilde{D}(H^{cop})$ are given by:
\begin{eqnarray*}
(f \otimes x)(g \otimes y) &=& g_1(x_3)g_3(S^{-1}x_1) (g_2f \otimes yx_2),\\
\Delta(f \otimes x) &=& (f_2 \otimes x_1) \otimes (f_1 \otimes x_2), {\mbox { and}}\\
S(f \otimes x) &=& f_1(S^{-1}x_3)f_3(x_1) (Sf_2 \otimes S^{-1}x_2).
\end{eqnarray*}
A direct check now shows that the map $S \otimes id_H : \tilde{D}(H)^{cop} \rightarrow \tilde{D}(H^{cop})$ is a Hopf algebra isomorphism.
\end{proof}

\section{Basic construction, crossed products and recognition}

This section is devoted to a few results that will be used in proving the uniqueness part of our main theorem. The emphasis of these results
is on the `basic construction' - the passage from a unital
algebra inclusion $A \subseteq B$ to the unital algebra
inclusion $B \subseteq C = End(B_A)$ (the algebra of right $A$-linear endomorphisms of $B$) where the inclusion
of $B$ in $C$ is via the left regular representation. 

Many of the results of this section are known - sometimes in greater generality - for Hopf-Galois extensions (in particular for twisted smash products) as in \cite{Kds2005,KrmTkc1981}, including crossed product recognition theorems as in \cite{DoiTkc1986,KdsNks2001}. Proofs are
included here only for completeness.


\begin{lemma}\label{relcomm}
Let $A \subseteq B$ be a unital inclusion of algebras with
 associated basic construction
$B \subseteq C$.
Then the centraliser algebras $B^A$ and $C^B$ are anti-isomorphic.
In particular, $A \subseteq B$ is irreducible if and only if $B \subseteq C$ is irreducible.
\end{lemma}

\begin{proof} The map $B^A \rightarrow C^B$ given by $b \mapsto \rho_b$
where $\rho_b(\tilde{b}) = \tilde{b}b$ is verified to be an anti-isomorphism.
\end{proof}


Before we prove the next theorem analysing the basic construction when $A \subseteq B$ is of the form $A \subseteq A \rtimes H$, we pause to observe the following.

\begin{lemma}\label{difficult}
Every linear map $H \rightarrow A \rtimes H$ is of the form $\lambda_{a \rtimes x} \beta_f$ for $a \otimes x \otimes f \in A \otimes H \otimes H^*$.
\end{lemma}

\begin{proof} Clearly, any such linear map is necessarily of the form
$z \mapsto a \otimes g(z)y$ for some $a \otimes y \otimes g \in A \otimes H \otimes H^*$.
Let $p$ be a left integral for $H^*$ and $h$ a left integral of $H$ with $p(h)=1$.
Let $a \otimes x \otimes f = a \otimes (gp_2)(h_2) yS^{-1}h_1 \otimes S^{-1}p_1$. Now, computation, using the properties of left integrals stated above Lemma \ref{commlemm} applied to both $h$ and $p$, shows that
$\lambda_{a \rtimes x}\beta_f(z) = a \otimes g(z)y$, as desired.
\end{proof}

\begin{theorem}\label{crp}
Suppose that $\alpha$ is an action of the finite-dimensional Hopf algebra $H$ on an algebra $A$ and
$B = A \rtimes H$. Let $B \subseteq C$ be the basic construction of $A \subseteq B$. Then, 
\begin{enumerate}
\item $C$ is isomorphic as an algebra over $B$ to $B \rtimes H^*$.
\end{enumerate}
%
%
If, further, the action $\alpha$ is outer, then 
\begin{enumerate}
\item[(2)]  $A^\prime \cap C ~( = End(_AB_A))  = H^*$ for the natural imbedding of $H^*$ in $C$, and
\item[(3)] $Hom(_AB_A,_AA_A)$ is 1-dimensional
and is identified in $H^*$ as the scalar multiples of a(ny) non-zero left
integral of $H^*$.
\end{enumerate}
\end{theorem}

\begin{proof} (1) 
Define a map $\theta : B \rtimes H^* \rightarrow C$ by
$\theta(b \rtimes f) = \lambda_b \circ \beta_f$ and note that this a well-defined map, i.e., is right $A$-linear, and, after a little calculation, is an
algebra homomorphism that restricts to the identity on $B$.

To see that $\theta$ is injective, take $Z = b\rtimes y \rtimes g \in B \rtimes H^* = A \rtimes H \rtimes H^*$ in $ker(\theta)$. To see that $Z=0$, it will suffice to see that  for arbitrary $f \in H^*$ and $x \in H$, $(id \otimes f \otimes x)(Z) = 0$.  Computation shows that $0 = \theta(Z)(1 \rtimes z) =
g(z_2)(b \rtimes yz_1)$, for all $z \in H$. Hence, for all $k \in H^*$ and $z \in H$, $g(z_2)k(yz_1)b = (id \otimes k_2(z_1)k_1 \otimes z_2)(Z) = 0$.
Now appeal to the well-known (and easily checked) fact that the map $k \otimes z \mapsto k_2(z_1) k_1 \otimes z_2$ of $H^* \otimes H$ to itself is invertible (with inverse $f \otimes x \mapsto f_2(S^{-1}x_1) f_1 \otimes x_2$) to produce $k \otimes z$ such that $k_2(z_1) k_1 \otimes z_2 = f \otimes x$, to finish the proof of injectivity.  

For surjectivity, first note that the map $x \otimes a \mapsto xa = (1 \rtimes x)(a \rtimes 1) = \alpha_{x_1}(a) \rtimes x_2 : H \otimes A \rightarrow A \rtimes H$ is a linear isomorphism with inverse given by $a \rtimes x \mapsto x_2 \otimes \alpha_{S^{-1}x_1}(a)$. It follows that any right $A$-linear map from $B$
to $B$ is determined by its action on elements of $H$. Now by Lemma \ref{difficult},
an arbitrary linear map from $H$ to $A \rtimes H$ can be expressed in the
form $\lambda_{a \rtimes x} \beta_f$ for $a \otimes x \otimes f \in A \otimes H \otimes H^*$. Since $\lambda_{a \rtimes x} \beta_f$ is right $A$-linear, surjectivity follows.

(2) Identify $C$ with $A \rtimes H \rtimes H^*$. Observe first that $1 \rtimes 1 \rtimes f$ commutes with $A$ for all $f \in H^*$. Conversely, suppose that $a \rtimes x \rtimes f \in A^\prime \cap C$. This implies that
$\tilde{a}a \rtimes x \rtimes f = a\alpha_{x_1}(\tilde{a}) \rtimes x_2 \rtimes f$ for all $\tilde{a} \in A$. Recalling that $a \rtimes x \rtimes f$ actually stands for a sum and comparing coefficients of a basis of $H^*$ on either side gives $\tilde{a}a \rtimes x = a\alpha_{x_1}(\tilde{a}) \rtimes x_2$ for all $\tilde{a} \in A$. This implies that $a \rtimes x \in A^\prime \cap B$ and is therefore a scalar by outerness of the action. Hence $a \rtimes x \rtimes f \in H^* \subseteq C$.

(3) $Hom(_AB_A,_AA_A)$ consists of those elements of $End(_AB_A)$ whose range is contained in $A$. Since $End(_AB_A) = \{\beta_f : f \in H^*\}$, we
need to see for what $f \in H^*$ is $\beta_f(A \rtimes H) \subseteq A$.
If $f =p$ - a non-zero left integral of $H^*$, then $\beta_f (a \rtimes x)
= a \rtimes p(x_2)x_1 = p(x)(a \rtimes 1)$, by the defining property of a left integral of $H^*$. On the other hand, if $\beta_f(A \rtimes H) \subseteq A$,
then, in particular,  $\beta_f(1 \rtimes x) = 1 \rtimes f(x_2)x_1 \in A$ for all $x \in H$. Thus $f(x_2)x_1$ must be a scalar multiple of $1_H$ for all $x \in H$ and applying $\epsilon$ shows that this scalar is necessarily $f(x)$. Thus $f$ must be a left integral of $H^*$.
\end{proof}

We omit the proof of the next proposition, the first three parts of which follow directly from Lemma \ref{relcomm} and Theorem \ref{crp}, while the fourth has a proof very similar to that of Theorem \ref{crp}(2).

\begin{proposition}\label{dbc}
Suppose that $\alpha$ is an outer action of the finite-dimensional Hopf algebra $H$ on an algebra $A$ and
$B = A \rtimes H$. Let $A \subseteq B \subseteq C \subseteq D$ be the iterated basic construction of $A \subseteq B$. Then, 
\begin{enumerate}
\item[(1)] $D$ is isomorphic as an algebra over $A$ to $A \rtimes H \rtimes H^* \rtimes H$,
\item[(2)] $B^\prime \cap D ~(= End(_BC_B)) = H$ for the natural imbedding of $H$ in $D$,
\item[(3)] $Hom(_BC_B,_BB_B)$ is 1-dimensional
and is identified in $H$ as the scalar multiples of a(ny) non-zero left
integral of $H$, and
\item[(4)]  $A^\prime \cap D ~( = End(_AC_B))  = H^* \rtimes H$ for the natural imbedding of $H^* \rtimes H$ in $D$. \qed
\end{enumerate}
\end{proposition}

%
%

Note that Proposition \ref{dbc}(4) implies that the multiplication map
$(A^\prime \cap C) \otimes (B^\prime \cap D) \rightarrow (A^\prime \cap D)$ is an isomorphism.
We now show that the crossed product by an outer action of a finite dimensional Hopf algebra recognizes the Hopf algebra. More precisely, we have the following theorem.

\begin{theorem}\label{unique}
Let $H$ be a finite dimensional Hopf algebra acting outerly on an algebra $A$. Then, the isomorphism class of the pair $A \subseteq A \rtimes H$ 
determines $H$ up to isomorphism, i.e., if $A \subseteq A \rtimes H \cong 
A \subseteq A \rtimes K$ as pairs of algebras, for some finite dimensional Hopf algebra $K$ acting outerly on $A$, then $H \cong K$ as Hopf algebras.
\end{theorem}

Before beginning the proof we note that by an isomorphism of pairs of algebras $A \subseteq B$ and $C \subseteq D$, we mean an algebra
isomorphism from $B$ to $D$ that restricts to an isomorphism from $A$ to $C$.

\begin{proof}[Proof of Theorem \ref{unique}] Begin with a pair of algebras $A \subseteq B$ known to be isomorphic to $A \subseteq A \rtimes H$. Perform the double basic construction to get the tower $A \subseteq B \subseteq C \subseteq D$
of algebras. It follows from Theorem \ref{crp}(2) and Theorem \ref{dbc}(2) that $A^\prime \cap C \cong H^*$ and $B^\prime \cap D \cong  H$ as algebras.

Now, Theorem \ref{crp}(3) and Proposition \ref{dbc}(3) give distinguished 1-dimensional subspaces $Hom(_AB_A,_AA_A) \subseteq A^\prime \cap C$ and
$Hom(_BC_B,_BB_B) \subseteq B^\prime \cap D$ that are identified with
the spaces of left integrals in $H^*$ and $H$ respectively.

Pick a non-zero element $p \in Hom(_AB_A,_AA_A)$. Since $p$ corresponds to a left-integral of $H^*$, for any $g \in H^* = A^\prime \cap C$, we
have $gp = g(1)p$. Thus we get the map $g \mapsto g(1): A^\prime \cap C = H^* \rightarrow k$ - which is the counit $\epsilon_{H^*}$ of $H^*$. Similarly, we get the map $\epsilon_H: B^\prime \cap D = H \rightarrow k$.

Finally, given arbitrary $f \in H^* = A^\prime \cap C$ and $x \in H = B^\prime \cap D$, consider $xf \in A^\prime \cap D$. Identifying $H^*$ and $H$ with their images in $H^* \rtimes H$, this is just the element
$(\epsilon \rtimes x)(f \rtimes 1) = \alpha_{x_1}(f) \rtimes x_2 = f_2(x_1)f_1 \rtimes x_2 \in H^* \rtimes H$. Pulling back this element via the natural isomorphism from 
$(A^\prime \cap C) \otimes (B^\prime \cap D)$ to $(A^\prime \cap D)$
gives the element $\alpha_{x_1}(f) \otimes x_2  \in H^* \otimes H$. Now applying $\epsilon_{H^*} \otimes \epsilon_H$ to this gives $f(x)$.

Thus, if $A \subseteq A \rtimes H \cong A \subseteq A \rtimes K$, we've seen that there are algebra isomorphisms $H \rightarrow K$ and
$H^* \rightarrow K^*$ that take the evaluation pairing between $H$ and $H^*$ to that between $K$ and $K^*$. This shows that the algebra
isomorphisms are bialgebra isomorphisms and therefore also Hopf algebra isomorphisms.\end{proof}

%

\section{The main theorem}

We are now ready to state our main result.

\begin{theorem}\label{main}
Let $H$ be a finite-dimensional Hopf algebra and $A \subseteq B$ be its derived pair. Then $B$ is isomorphic, as an algebra over $A$,  to $A \rtimes L$ (for some (outer) action of $L = D(H)^{cop}$ on $A$) and further, up to isomorphism, $L$ is the only finite-dimensional Hopf algebra with this property.
\end{theorem}

We briefly sketch the proof of this theorem before going into the details. We first exhibit $L = \tilde{D}(H)^{cop} \cong D(H)^{cop}$ (see \S 3) as a subalgebra of $B$,  show that the multiplication map $A \otimes L$ to $B$
is an isomorphism and that $A$ is stable under the `adjoint action' of $L$. This suffices to see that $B$ is isomorphic as an algebra over $A$ to $A \rtimes L$. The uniqueness of $L$ needs
a little more effort for which we show that from the pair
$A \subseteq A \rtimes H$ (for any outer action of a Hopf algebra $H$ on $A$) we can first recover the algebras
$H$ and $H^*$ as relative commutants and then also the
natural evaluation pairing between them, and therefore the Hopf algebra structure on $H$.


While the following lemma is quite easy to prove, deriving the form of the
homomorphism from $L$ to $B$ took us the longest time and involved
application of the diagrammatics of Jones' planar algebras. Having obtained
the formula though, verification is simple.

\begin{lemma}\label{injective}
The map $L \rightarrow B$ defined as  the composite map $L \rightarrow H^{[0,2]} \rightarrow B$ where $L \rightarrow H^{[0,2]} =  H^* \rtimes H \rtimes H^*$ is defined by
$$
(f \otimes x) \mapsto f_1(Sx_1) f_3 \rtimes Sx_2 \rtimes f_2.
$$
 is
an injective algebra homomorphism. \qed
\end{lemma}

\begin{proof} We omit the verification that the map defined is an algebra homomorphism. To see that it is injective, we consider the map $H^{[0,2]} \rightarrow L$ defined by $f \rtimes x \rtimes g \mapsto f(1)g_1(S^{-1}x_2) (g_2 \otimes S^{-1}x_1)$ and verify that it is a left inverse.
\end{proof}

\begin{remark}
In particular, Lemma \ref{injective} implies that $L$ is a subalgebra of $H^* \rtimes H \rtimes H^* \rtimes H$ which is a matrix algebra of size
$dim(H)^2$ and also is the tensor square of the Heisenberg double
$H^* \rtimes H$ of $H^*$. This is one of the results of \cite{Ksh1996}.
\end{remark}

We will identify $L$ with its image in $B$. Note that under this identification, $(f \otimes 1) \mapsto f_2 \rtimes 1 \rtimes  f_1 \in H^{[0,2]}$ and $(\epsilon \otimes x) \mapsto Sx \in H^1$.

\begin{lemma}\label{mult}
The multiplication map of $B$ restricted to $A \otimes L$ is an isomorphism.
\end{lemma}

\begin{proof}
Given $\cdots \rtimes x^{-1} \rtimes \epsilon \rtimes 1 \rtimes f^2\cdots \in A$ and $f \otimes x \in L$ (identified with $f_1(Sx_1) f_3 \rtimes Sx_2 \rtimes f_2 \in H^{[0,2]} \subseteq B$), their product is computed to be:
$$
f_1(Sx_1)f^2_1(Sx_2)f_3(x^3_1) (\cdots \rtimes f^{-2} \rtimes x^{-1} \rtimes f_4 \rtimes Sx_3 \rtimes f^2_2f_2 \rtimes x^3_2 \rtimes f^4 \rtimes x^5 \rtimes \cdots)
$$
Thus, to prove the lemma, it suffices to verify that the map 
$$
g \otimes y \otimes f \otimes x \mapsto f_1(Sx_1)g_1(Sx_2)f_3(y_1) ( f_4 \otimes Sx_3 \otimes g_2f_2 \otimes y_2)
$$
of $H^* \otimes H \otimes H^* \otimes H$ to itself is a linear isomorphism.
We assert, and omit the straightforward but very computational proof, that
the map
$$
p \otimes z \otimes q \otimes w \mapsto p_1(Sw_1)q_1(S^{-1}z_2) (q_2Sp_2 \otimes w_2 \otimes p_3 \otimes S^{-1}z_1)
$$
is its inverse.
\end{proof}

\begin{proposition}\label{actionona}
The map $\gamma: L  \rightarrow End(B)$ given
by $\gamma_{(f \otimes x)}(b) = (f \otimes x)_1 b S((f \otimes x)_2)$ maps $A$ to itself.
\end{proposition}

\begin{proof} The map $\gamma: L  \rightarrow End(B)$ 
is easily verified to be an action of $L$ on $B$ and so it suffices to check for $f \in H^*$, $x \in H$ and
$a \in A$, that 
$\gamma_{(f \otimes 1)}(a), \gamma_{(\epsilon \otimes x)}(a) \in A$.

Taking $a = \cdots \rtimes x^{-3} \rtimes f^{-2} \rtimes x^{-1} \rtimes \epsilon \rtimes 1 \rtimes f^2 \rtimes x^3 \rtimes f^4 \rtimes x^5 \rtimes \cdots$, we compute
\begin{eqnarray*}
\gamma_{(f \otimes 1)}(a)&=&f_2(x^{-1}_2)f_3(Sx^3_1) (\cdots \rtimes f^{-2} \rtimes x^{-1}_1 \rtimes \epsilon \rtimes 1 \rtimes f_1f^2Sf_4 \rtimes x^3_2 \rtimes f^4  \rtimes \cdots),\\
\gamma_{(\epsilon \otimes x)}(a)&=&f^2_1(x) (\cdots \rtimes f^{-2} \rtimes x^{-1} \rtimes \epsilon \rtimes 1 \rtimes f^2_2 \rtimes x^3 \rtimes f^4  \rtimes \cdots),
\end{eqnarray*}
both of which are clearly in $A$.
%
%
\end{proof}

\begin{proof}[Proof of Theorem \ref{main}]
The hypotheses of Lemma \ref{simple} are satisfied by $A \subseteq B$
and $L$, by Lemma \ref{mult} and Proposition \ref{actionona}. Thus, by its conclusion, $B$ is isomorphic as an algebra over $A$ to $A \rtimes L$, for some action of $L$ on $A$. This action is outer since $A \subseteq B$ is
irreducible by Proposition \ref{commutants}. Finally, by Theorem \ref{unique},
$L$ is unique up to isomorphism.
%
%
%
%
%
%
\end{proof}

The result asserted in our abstract is an easy corollary using Lemma \ref{hop}.

\begin{corollary}
Let $A \subseteq B$ be the derived pair of $H^{cop}$ ($\cong H^{op}$) for a finite-dimensional
Hopf algebra $H$. Then $B$ is isomorphic as an algebra over $A$ to
$A \rtimes D(H)$ for some outer action of $D(H)$ on $A$ and further
$D(H)$ is the only finite-dimensional Hopf algebra with this property.
\end{corollary}

\begin{remark}
Scattered throughout this note are various explicit calculations of inverses
(sometimes one-sided)
of maps of tensor powers of $H$ - for instance, in the proofs
of Lemmas \ref{difficult}, \ref{injective} and \ref{mult}. We hope to address the problem of finding general pictorial methods of doing so in a future publication.
\end{remark}

\end{document}